\documentclass[11pt]{article}
\usepackage[cm]{fullpage}

\usepackage{amscd, eucal}
\usepackage{amsmath}
\usepackage{amsfonts}

\newcommand{\de}{\partial}

\newcommand{\ov}[1]{\overline{#1}}

\newcommand{\mn}{\sqrt{-1}}

\newcommand{\tr}[2]{\textrm{tr}_{#1}{#2}}
\newcommand{\ti}[1]{\tilde{#1}}

\begin{document}
\newcounter{remark}
\newcounter{theor}
\setcounter{remark}{0}
\setcounter{theor}{1}
\newtheorem{claim}{Claim}
\newtheorem{theorem}{Theorem}[section]
\newtheorem{proposition}[theorem]{Proposition}
\newtheorem{lemma}[theorem]{Lemma}
\newtheorem{defn}{Definition}[theor]
\newtheorem{corollary}[theorem]{Corollary}
\newenvironment{proof}[1][Proof]{\begin{trivlist}
\item[\hskip \labelsep {\bfseries #1}]}{\end{trivlist}}
\newenvironment{remark}[1][Remark]{\addtocounter{remark}{1} \begin{trivlist}
\item[\hskip
\labelsep {\bfseries #1  \thesection.\theremark}]}{\end{trivlist}}

\centerline{\Large \bf The Calabi-Yau equation}
\smallskip
\centerline{\Large \bf on the Kodaira-Thurston manifold\footnote{Research supported in part by National Science Foundation grant DMS-08-48193.  The second-named author is also supported in part by a Sloan Foundation fellowship.}}

\bigskip
\centerline{\bf Valentino Tosatti and Ben Weinkove}

\bigskip

\begin{abstract}   
We prove that the Calabi-Yau  equation can be solved on the Kodaira-Thurston manifold for all given $T^2$-invariant volume forms.  This provides support for Donaldson's conjecture that Yau's theorem has an extension to symplectic four-manifolds with compatible but non-integrable almost complex structures.
\end{abstract}

\section{Introduction}

A fundamental property of a compact K\"ahler manifold $(M^n, \omega)$ is that one can find K\"ahler metrics with prescribed volume form in a fixed K\"ahler class.  This is known as Yau's Theorem \cite{Y}.  More precisely, given a K\"ahler class $\kappa$ and a volume form $\sigma$  with $\int_M \sigma = \kappa^n$, there exists a unique K\"ahler form $\tilde{\omega}$ in $\kappa$ solving 
\begin{equation} \label{CY0}
\tilde{\omega}^n = \sigma.
\end{equation}
We call (\ref{CY0}) the Calabi-Yau equation.  

In \cite{D}, Donaldson conjectured that Yau's theorem can be extended to the case of general symplectic four-manifolds with compatible almost complex structures, at least in the case $b^+=1$. Moreover, Donaldson outlined a program to use estimates for the Calabi-Yau equation and its generalizations to prove new results for four-manifolds.   For a detailed discussion of this program and recent developments, we refer the reader to \cite{D, W, TWY, LZ, TW, DLZ}.   It was shown in \cite{W} and \cite{TWY} that many of Yau's estimates for (\ref{CY0}) carry over to the non-K\"ahler setting.  In particular, the Calabi-Yau equation can be solved if the Nijenhuis tensor of the almost complex structure is small in a certain sense \cite{W} or if a curvature condition holds for the fixed metric \cite{TWY}.   In this paper we investigate (\ref{CY0}) in the case of a  well-known four-manifold: the Kodaira-Thurston manifold.

The Kodaira-Thurston manifold is given by $M = S^1 \times ( \textrm{Nil}^3 / \Gamma)$ where $\textrm{Nil}^3$ is the Heisenberg group
$$\textrm{Nil}^3 = \left\{ A \in GL(3, \mathbb{R}) \ \Bigg| \ A = \begin{pmatrix} 1 & x & z \\ 0 & 1 & y \\ 0 & 0 & 1 \end{pmatrix}, \ x, y, z \in \mathbb{R}  \right\},$$
and $\Gamma$ is the subgroup of $\textrm{Nil}^3$ consisting of those elements of $\textrm{Nil}^3$ with integral entries, acting by left multiplication. 
Kodaira first investigated $M$ in the 1950s, showing that $M$ admits an integrable complex structure \cite{K}.  Thurston \cite{Th}  later observed that $M$ also admits a symplectic form but no K\"ahler structure, the first manifold known to have this property. This manifold and its higher-dimensional generalizations have been thoroughly studied over the years, see for example \cite{A, AD, CFG, FGG, M}. In this paper we will show that the Calabi-Yau equation can be solved on $M$ assuming $T^2$ invariance.  We make use of some ideas and estimates from \cite{W} and \cite{TWY}.

Writing $t$ for the $S^1$ coordinate, the $1$-forms $dx$, $dt$, $dy$ and $dz- xdy$ on $S^1 \times \textrm{Nil}^3$ are invariant under the action of $\Gamma$ and thus define $1$-forms on $M$.    One can use these $1$-forms to define a symplectic form
 $$\Omega = dx \wedge dt + dy \wedge (dz- xdy)$$
 on $M$ and a compatible almost complex structure
 $$J(dx) = dt, \quad J(dy) = dz- xdy.$$
 The data $(M, \Omega, J)$ is thus an \emph{almost-K\"ahler} manifold, but is not K\"ahler since $J$ is not an integrable complex structure.
 Indeed the Kodaira-Thurston manifold cannot admit K\"ahler structures because  $b^1(M)=3$  \cite{Th}.

There is  a  $T^2$-action on the Kodaira-Thurston manifold $M$ which preserves $\Omega$ and $J$.  Indeed, if we let $S^1$ and $\mathbb{R}$ act on  $S^1 \times \textrm{Nil}^3$ by translation in the $t$ and $z$ coordinates respectively then this action commutes with $\Gamma$ and  gives a free $T^2$ action on $M$ preserving the $1$-forms $dx$, $dt$, $dy$ and $dz-xdy$. This is essentially the only free symplectic $T^2$ action on $M$ \cite{G}.

Our main result is as follows.

\begin{theorem} \label{maintheorem}
Let $\sigma$ be a smooth volume form on $M$, invariant under the $T^2$-action given above and normalized so that $\int_M \sigma = \int_M \Omega^2$.  Then there exists a unique $T^2$-invariant symplectic form $\tilde{\omega}$ cohomologous to $\Omega$, compatible with the almost complex structure $J$, solving the Calabi-Yau equation:
\begin{equation} \label{CY1}
\tilde{\omega}^2 = \sigma.
\end{equation}
\end{theorem}

The uniqueness part of Theorem \ref{maintheorem} is due to Donaldson \cite{D} (see also \cite{W}).  Note that  $T^2$-invariance  is not required for the uniqueness statement.  On any $4$-manifold equipped with an almost complex structure $J$,  if $\tilde{\omega}_1$ and $\tilde{\omega}_2$ are  cohomologous symplectic forms compatible with $J$ and satisfying $\tilde{\omega}_1^2=\tilde{\omega}_2^2$ then $\tilde{\omega}_1=\tilde{\omega}_2$.

In addition, one can recast Theorem \ref{maintheorem} in terms of the Ricci form of the canonical connection. We explain this in Section \ref{Ricci} - see Theorem \ref{theoremricci} below.  This can be regarded as a kind of analogue of another formulation of Yau's theorem often referred to as the Calabi conjecture:  any representative of the first Chern class of a K\"ahler manifold can be written as the Ricci curvature of a K\"ahler metric in a given K\"ahler class.

The outline of the rest of the paper is as follows.  First, in Section \ref{sectionbound}, we prove the key \emph{a priori} estimate for the Calabi-Yau equation on the Kodaira-Thurston manifold.   Part of this argument involves a Moser iteration argument from \cite{TWY}.  We complete the proof of Thereom \ref{maintheorem} in Section \ref{proofmain}, using some estimates from \cite{W, TWY} and also \cite{D}.   In Section \ref{questions} we end with some further remarks and questions.

\setcounter{equation}{0}
\section{An \emph{a priori} estimate for the Calabi-Yau equation} \label{sectionbound}

In this section we derive a uniform \emph{a priori} bound for a $T^2$-invariant solution $\tilde{\omega}$ of the Calabi-Yau equation
\begin{equation} \label{CY2}
\tilde{\omega}^2 = \sigma,
\end{equation}
where $\sigma$ is a fixed volume form and $\tilde{\omega}$ is compatible with $J$.  We write $\sigma = e^F\Omega^2$ for a smooth $T^2$-invariant function $F$ so that the 
Calabi-Yau equation becomes
\begin{equation} \label{CY3}
\tilde{\omega}^2 = e^F \Omega^2.
\end{equation}
By the normalization of $\sigma$ we have 
\begin{equation} \label{sigma}
\int_M e^F \Omega^2 = \int_M \Omega^2.
\end{equation}

We now introduce some natural objects associated to the almost-K\"ahler manifold $(M, \Omega, J)$.  Write $g$ for the almost-K\"ahler metric, given by
$$g(X,Y) = \Omega(X, JY),$$
where we let $J$ act on vectors by duality with the following convention:
if $\tau$ is a $1$-form we let $\tau(JX)=-(J\tau)(X)$. We also let $J$ act on 
 $2$-forms $\eta$ by $(J\eta)(X,Y) = \eta(JX, JY)$.

 We assume that the solution $\tilde{\omega}$ of the Calabi-Yau equation
 $$\tilde{\omega}^2 = e^F \Omega^2$$ is cohomologous to $\Omega$ and so we can write
 \begin{equation} \label{tildeomega}
\tilde{\omega} =  \Omega  + da,
\end{equation}
for  $a$ a 1-form. Since $\ti{\omega}$ and $\Omega$ are $T^2$-invariant, after averaging $a$ by the $T^2$-action, we may assume that $a$ is also $T^2$-invariant.

The solution $\tilde{\omega}$ of (\ref{CY2}) is compatible with $J$ and we write 
$\tilde{g}$ for the associated almost-K\"ahler metric, given by
$$\tilde{g}(X,Y) = \tilde{\omega}(X,JY).$$
In this section we prove the following \emph{a priori} bound on the metric $\tilde{g}$.

\begin{theorem} \label{mainestimate}
There is a uniform constant $C$ depending only on $\inf_M \Delta F$  such that
\begin{equation}
\emph{\textrm{tr}}_g{\tilde{g}} \le C,
\end{equation}
where $\Delta$ is the Laplace operator associated to $g$.
\end{theorem}
\begin{proof}

A general $T^2$-invariant $1$-form $a$ can be written
$$a= f_1 dx + f_2 dt + f_3 dy + f_4 (dz - xdy),$$
where $f_i=f_i(x,y)$ (for $i=1, \ldots, 4$).  Taking the exterior derivative:
\begin{eqnarray*}
\lefteqn{da = f_{2,x} dx \wedge dt +   (f_{3,x} - f_{1,y} -f_4 )dx \wedge dy - f_{2,y} dt \wedge dy} \\
&& + f_{4,x} dx \wedge (dz-xdy) + f_{4,y} dy \wedge (dz-xdy),
 \end{eqnarray*}
where here and henceforth letter subscripts denote partial derivatives.  Compute
\begin{eqnarray} \nonumber
J(da) &   = &  f_{2,x} dx \wedge dt +   (f_{3,x} - f_{1,y} -f_4 )dt\wedge (dz-xdy)  \\
&& \mbox{} + f_{2,y} dx \wedge (dz-xdy)  - f_{4,x} dt \wedge dy + f_{4,y} dy \wedge (dz-xdy).
 \end{eqnarray}
The condition that $\ti{\omega}$ of the form (\ref{tildeomega}) is compatible with $J$ is equivalent to the equation
\begin{equation}
J(da)=da,
\end{equation}
and by the above this reduces to the following system of differential equations:
 \begin{eqnarray} \label{Ida3}
 f_{3,x} - f_{1,y} - f_4 & = & 0 \\ \label{Ida4}
 f_{2,y} - f_{4,x} & = & 0.
 \end{eqnarray}
Thus we can write
\begin{eqnarray} \nonumber
\ti{\omega} & = & (1+f_{2,x})dx\wedge dt+ (1+f_{4,y})dy\wedge (dz-xdy)
\\ && \mbox{} + f_{4,x} dx\wedge (dz-xdy)- f_{4,x}dt\wedge dy,\end{eqnarray}
and
 \begin{equation}
\ti{\omega}^2 = \left\{ (1+f_{2,x})(1+f_{4,y}) - f_{4,x}^2 \right\} \Omega^2, \label{square1}
 \end{equation}
and hence the Calabi-Yau equation (\ref{CY3}) becomes
\begin{equation}
(1 + f_{2,x})(1+f_{4,y}) -f_{4,x}^2 = e^F. \label{CY4}
\end{equation}

The basis of left-invariant vector fields dual to $\{dx, dt, dy, dz-xdy\}$ 
is $\{\de_x, \de_t, \de_y+x\de_z, \de_z\}$.
The matrix of $g$ with respect to this basis is the identity, while the matrix of $\ti{g}$ is
\begin{equation} \label{tildeg1}
\tilde{g} = \begin{pmatrix} 1+f_{2,x} & 0 & f_{4,x} & 0 \\ &&&\\ 0 & 1 +f_{2,x} & 0 & f_{4,x} \\ &&&\\ f_{4,x} & 0 & 1 +f_{4,y} & 0 \\&&&\\0 & f_{4,x} & 0 & 1 +f_{4,y} \end{pmatrix}.
\end{equation}

The following  lemma is the key ingredient of this paper and makes crucial use of the structure of the Kodaira-Thurston manifold.

\begin{lemma} \label{lemmalap}  Let $\tilde{\Delta}$ be the Laplace operator  associated to $\tilde{g}$ and define
\begin{equation}
u = \frac{ \emph{\textrm{tr}}_g{\tilde{g}}}{2} = 2 + f_{2,x} + f_{4,y}.
\end{equation}
Then
%There exists a uniform constant $C$ depending only on $\inf_M\Delta F$ such that $u=\frac{1}{2}\emph{\textrm{tr}}_g{\tilde{g}}$ satisfies
\begin{equation}
\tilde{\Delta}u \ge \inf_M \Delta F.
\end{equation}
\end{lemma}

\begin{proof}[Proof of Lemma \ref{lemmalap}]
A straightforward calculation shows that the Laplace operators of $g$ and $\tilde{g}$ respectively applied to  a general $T^2$-invariant function $\psi = \psi(x,y)$ are given by the formulae
\begin{eqnarray} \label{lap}
\Delta \psi & = & \frac{2\Omega \wedge d(Jd \psi)}{\Omega^2}  =  \psi_{xx} +\psi_{yy} 
\end{eqnarray}
and
\begin{eqnarray} \label{tildelap}
\tilde{\Delta} \psi & = & \frac{2 \tilde{\omega} \wedge d(Jd \psi)}{\tilde{\omega}^2} = \frac{1}{\nu} \left( (1+f_{2,x}) \psi_{yy} + (1 +f_{4,y}) \psi_{xx} - 2f_{4,x} \psi_{xy} \right), \quad 
\end{eqnarray}
where $\nu = (1+f_{2,x})(1+f_{4,y}) - f_{4,x}^2=e^F$.

Applying (\ref{lap}) and (\ref{tildelap}) to $\log \nu$ and  $u$ respectively, and making use of \eqref{Ida4} we find
\begin{equation} \label{tildedeltau}
\tilde{\Delta} u = \Delta \log \nu + \frac{1}{\nu} \left( \frac{\nu_x^2}{\nu} + \frac{\nu_y^2}{\nu} + 2\left(-f_{2,xx} f_{2,yy} + f^2_{2,yx} - f_{4,yy}f_{4,xx} +f^2_{4,yx} \right) \right).
\end{equation}
Now
\begin{equation}
\nu = AB-D^2 , \quad \nu_x = A f_{2,yy} + B f_{2,xx} - 2D f_{2,yx}, \quad \nu_y = A f_{4,yy} + B f_{4,xx} - 2 D f_{4,xy},
\end{equation}
where $A = 1 +f_{2,x}$, $B=1+ f_{4,y}$ and $D = f_{4,x}$.   Thus 
$$\nu_x = \textrm{tr} (\mathcal{L}), \quad \nu_y = \textrm{tr} (\mathcal{M}).$$
where
$$\mathcal{L} =  \begin{pmatrix} A & -D \\ -D & B \end{pmatrix} \begin{pmatrix} f_{2,yy} & f_{2,xy} \\ f_{2,xy} & f_{2,xx}  \end{pmatrix}, \quad \textrm{and} \quad \mathcal{M} = \begin{pmatrix} A & -D \\ -D & B \end{pmatrix} \begin{pmatrix} f_{4,yy} & f_{4,xy} \\ f_{4,xy} & f_{4,xx} \end{pmatrix}.$$
On the other hand,
$$\nu (f_{2,xx} f_{2,yy} - f^2_{2,yx}) = \det (\mathcal{L}), \quad \nu (f_{4,xx} f_{4,yy} - f^2_{4,yx}) = \det (\mathcal{M}),$$
and so
\begin{eqnarray} \nonumber
\lefteqn{\nu_x^2 + \nu_y^2 + 2\nu \left(-f_{2,xx} f_{2,yy} + f^2_{2,yx} - f_{4,yy}f_{4,xx} +f^2_{4,yx} \right)} \\ \label{nonneg}
& = & (\tr{} (\mathcal{L}))^2 - 2\det (\mathcal{L}) +(\tr{} (\mathcal{M}))^2 - 2\det (\mathcal{M})\ge 0.
\end{eqnarray}
For the inequality of (\ref{nonneg}), we are using the following elementary fact from linear algebra.  If $P$ and $Q$ are $2 \times 2$ symmetric matrices with $P$ positive definite, then
\begin{equation} \label{la}
(\tr{} (PQ))^2 - 2\det (PQ) \ge 0.
\end{equation}
Indeed, a direct computation gives (\ref{la})  in the case when $P$ is the identity matrix.  For general $P$, write $P=SS^T$.  Then
\begin{equation} \label{la2}
(\tr{} (PQ))^2 - 2\det (PQ) = (\tr{} (S^TQS))^2 - 2\det (S^TQS)   \ge 0,
\end{equation}
since $S^TQS$ is symmetric, thus establishing (\ref{la}).

Combining (\ref{tildedeltau}) and (\ref{nonneg}) completes the proof of Lemma \ref{lemmalap}.
 Q.E.D. \end{proof}

We can now finish the proof of Theorem \ref{mainestimate}.
Observe that 
$$u=\frac{2\Omega\wedge\ti{\omega}}{\Omega^2},$$
and thus
$$\int_M u \Omega^2=2\int_M \Omega \wedge (\Omega +da)  =2 \int_M\Omega^2.$$
From this $L^1$ bound of $u$ together with Lemma  \ref{lemmalap} we can apply 
the Moser iteration argument of \cite[Theorem 1.4]{TWY} to obtain $\| u\|_{C^0} \leq C$, where the constant $C$ depends only on $F$.  Although this argument is contained in \cite{TWY}, we include a brief sketch here for the reader's convenience.  For $p>0$, compute using the Calabi-Yau equation:
\begin{eqnarray}  \nonumber
\int_M | \nabla_g u^{p/2}|^2 \Omega^2 & \le & C' \int_M u d(u^{p/2}) \wedge J d(u^{p/2}) \wedge \tilde{\omega} \\ \nonumber
& = &  -  \frac{C' p}{8} \int_M u^p (\tilde{\Delta} u) \tilde{\omega}^2 \\ \label{nablau} 
& \le & C'' p \int_M u^p \Omega^2,
\end{eqnarray}
where we have used an integration by parts to go from the first to the second line, and Lemma \ref{lemmalap} for the third line.  Combining (\ref{nablau}) with the Sobolev inequality applied to $u^{p/2}$ we obtain for any $p>0$,
$$ \| u \|_{L^{2p}} \le C^{1/p} p^{1/p} \| u \|_{L^p},$$
and a straightforward iteration argument now gives
$$\| u \|_{C^0} \le C \| u \|_{L^1} \le C,$$
as required.
 Q.E.D.
\end{proof}

\setcounter{equation}{0}
\section{Proof of the Main Theorem} \label{proofmain}

We now complete the proof of the main theorem.  Following the method proposed in  \cite{D, W} we solve the Calabi-Yau equation (\ref{CY2}) using a continuity method. 

For $t \in [0,1]$, consider the family of equations
\begin{equation} \label{cty}
\tilde{\omega}_t^2 = e^{tF+c_t} \Omega^2, \quad \textrm{with } \, [\tilde{\omega}_t] = [\Omega],
\end{equation}
where the symplectic form $\ti{\omega}_t$ is compatible with $J$ and $c_t$ is the constant given by
\begin{equation}
\int_M e^{tF+c_t} \Omega^2 = \int_M \Omega^2.
\end{equation}
We wish to show that (\ref{cty}) has a $T^2$ invariant solution for $t=1$.  
Consider the set
$$\mathcal{T} = \{ t \in [0,1] \ | \ \textrm{there exists a smooth solution of (\ref{cty}) for } t' \in [0,t] \}.$$
Since $\tilde{\omega}_0 = \Omega$ solves (\ref{cty}) for $t=0$ we see that $0 \in \mathcal{T}$.  To prove the main theorem it suffices to show that $\mathcal{T}$ is both open and closed in $[0,1]$. 

For the openness part, we first need a brief discussion on the cohomology of $M$.  First observe that 
the Kodaira-Thurston manifold has $b^+(M)=2$ (see, for example, equation (3.1) of \cite{L}).  A basis 
for the space of $g$-harmonic self-dual $2$-forms is given by $\Omega$ together with the symplectic form $\Omega_1$ given by
\begin{equation} \label{Omega1}
\Omega_1 = dx \wedge (dz-xdy) + dt \wedge dy.
\end{equation}
Notice that $\Omega_1$ is $T^2$-invariant, closed, of type $(2,0)+(0,2)$ and self-dual.
In particular $J(\Omega_1)=-\Omega_1$.

In \cite{LZ} a cohomology group $H^-_J(M)$ was introduced as the space of all cohomology classes in $H^2(M;\mathbb{R})$ that can be represented by closed forms of type $(2,0)+(0,2)$.  This was further studied in \cite{DLZ}.  In our case $H^-_J(M)$ is $1$-dimensional, generated by $[\Omega_1]$.

We also have that
$$\Omega_1\wedge\Omega=0, \quad \Omega_1^2=\Omega^2,$$
and so $\Omega$, $\Omega_1$ span a maximal subspace $H^+\subset H^2(M;\mathbb{R})$ on which the intersection form is positive definite.
Proposition 1 of \cite{D} (cf. \cite{W}) shows that if there exists a solution of (\ref{cty}) for $t_0 \in [0,1]$ then one can find a solution $\tilde{\omega}_t$ of the equation
\begin{equation} \label{cty2}
\tilde{\omega}_t^2 = e^{tF+c_t} \Omega^2, 
\end{equation}
 for $t$ sufficiently close to $t_0$, with $\tilde{\omega}_t$ lying in the subspace $H^+$.   We claim that $\tilde{\omega}_t$ lies in $[\Omega]$.  Indeed, writing $\tilde{\omega}_t = \alpha_t \Omega + \beta_t \Omega_1+da$ we see that
\begin{equation} \label{deqn1}
\int_M \Omega^2 = \int_M \tilde{\omega}_t^2 =\alpha_t^2 \int_M \Omega^2 + \beta_t^2 \int_M \Omega_1^2,
\end{equation}
and
\begin{equation}  \label{deqn2}
 0 = \int_M \tilde{\omega}_t \wedge \Omega_1 = \beta_t \int_M \Omega_1^2,
\end{equation}
giving $\alpha_t=1$ and $\beta_t=0$.  For (\ref{deqn2}) we have used the fact that $\tilde{\omega}_t$ is of type (1,1) and $\Omega_1$ is of type $(2,0)+(0,2)$. Hence $\tilde{\omega}_t$ lies in $[\Omega]$, showing that the set
  $\mathcal{T}$ is open.  More generally the same argument applies to all $4$-manifolds satisfying $\dim H^-_J(M)=b^+(M)-1$ (see \cite{LZ, DLZ, li}). Moreover, since $\Omega$ and $F$ have $T^2$ symmetry, the implicit function theorem argument of \cite{D} shows that the solution $\tilde{\omega}_t$ for $t \in \mathcal{T}$ must have $T^2$ symmetry.  To show that $\mathcal{T}$ is closed it remains to prove that a solution $\tilde{\omega}_t$ of (\ref{cty}) is uniformly bounded in $C^{\infty}$, independent of $t$.  

For convenience, we write $F$ for $tF+c_t$ and $\tilde{\omega}$ for $\tilde{\omega}_t$.  The symplectic form $\tilde{\omega}$ is of the form (\ref{tildeomega}).  
 Then the result of Theorem \ref{mainestimate} shows that $\tr{g}{\tilde{g}}$ is bounded by a constant depending only on the $C^2(g)$ bound of $F$.  We can now directly apply the argument of \cite{W} or 
 \cite{TWY} to obtain a uniform H\"older bound on the solution $\tilde{\omega}$.  The higher order estimates then follow from the argument given in \cite{D} or \cite{W} (see also \cite{TWY}).  This completes the proof of the main theorem.

\setcounter{equation}{0}
\section{The Ricci form of the canonical connection} \label{Ricci}

We now show that the main theorem can be recast in terms of the Ricci form of a certain connection on $M$.
 In general,  given any symplectic form $\omega$ compatible with $J$ there is an associated \emph{canonical connection} $\nabla$ on $M$.  This connection is uniquely determined by the properties that if $g$ is the associated almost-K\"ahler metric then $\nabla g= 0 =\nabla J$ and the $(1,1)$-part of the torsion of $\nabla$ vanishes identically (see, for example, \cite{TWY} and the references therein). The curvature form of this connection expressed with respect to a local unitary frame is a skew-Hermitian matrix of $2$-forms $\{\Psi^i_j\}$, ($i,j =1,2$). The $2$-form
$$\mathrm{Ric}(\omega, J)=\frac{\sqrt{-1}}{2\pi}\sum_{i=1}^2\Psi^i_i,$$
is then closed and cohomologous to the first Chern class $c_1(M,J)$.  We call this $2$-form the Ricci form of the canonical connection.

Then the main theorem can be restated as follows\footnote{The idea of reformulating the Calabi-Yau equation in terms of the Ricci form of the canonical connection was known to V. Apostolov and T. Dr\u{a}ghici shortly after the paper \cite{TWY} appeared (see the discussion in \cite{li}).}  (cf. \cite[Conjecture 2.4]{TW} or Question 6.8 of \cite{li}).

\begin{theorem} \label{theoremricci}
Let $F$ be a smooth $T^2$-invariant function on $M$. Then there exists a $T^2$-invariant symplectic form $\tilde{\omega}$ on $M$, compatible with the almost complex structure $J$, satisfying
\begin{equation} \label{CY5}
\mathrm{Ric}(\ti{\omega},J) = -\frac{1}{2}d(JdF).
\end{equation}
\end{theorem}
\begin{proof}
We choose $\ti{\omega}$ to be the solution of the Calabi-Yau equation \eqref{CY3} given by Thereom \ref{maintheorem}.
Differentiating twice the logarithm of \eqref{CY3} gives
$$\mathrm{Ric}(\ti{\omega},J)=-\frac{1}{2}d(JdF)+\mathrm{Ric}(\Omega,J),$$
as in \cite[(3.16)]{TWY} (and see also \cite[pag.72]{GH}, \cite[Proposition 4.5]{li}). It remains to show that $\mathrm{Ric}(\Omega,J)=0$.  We choose the global left-invariant unitary coframe 
$$\theta^1=\frac{dx+\mn dt}{\sqrt{2}},\quad \theta^2=\frac{dy+\mn(dz-xdy)}{\sqrt{2}}.$$
We first claim that the connection $1$-forms $\{\theta^i_j\}$ ($i,j =1,2$) of the canonical connection of $g$ are given by
$$\theta^1_1=0,\quad  \theta^1_2=-\frac{\mn}{2\sqrt{2}}\theta^2, \quad \theta^2_1=-\frac{\mn}{2\sqrt{2}}\ov{\theta^2}, \quad \theta^2_2=\frac{\mn}{2\sqrt{2}}(\theta^1+\ov{\theta^1}).$$
Indeed, the matrix $\{\theta^i_j\}$ is skew-Hermitian, and so defines a connection $\nabla$ with $\nabla g=0=\nabla J$. The torsion $2$-forms $\{\Theta^i\}$ ($i=1,2$) of $\nabla$ are defined by the first structure equation
$$d\theta^i=-\theta^i_j\wedge\theta^j+\Theta^i.$$
Since we have
$$d\theta^1=0,\quad d\theta^2=-\frac{\mn}{\sqrt{2}}dx\wedge dy=
-\frac{\mn}{2\sqrt{2}}(\theta^1\wedge\ov{\theta^2}-\theta^2\wedge\ov{\theta^1}
+\theta^1\wedge\theta^2+\ov{\theta^1}\wedge\ov{\theta^2}),$$
one readily sees that
$$\Theta^1=0,$$
$$\Theta^2=-\frac{\mn}{2\sqrt{2}}\ov{\theta^1}\wedge\ov{\theta^2},$$
which have no $(1,1)$-part, proving the claim. We note here that since $\Omega$ is closed, the torsion $\{\Theta^i\}$ is equal to the Nijenhuis tensor of $J$ \cite{TWY}.

The curvature $\{\Psi^i_j\}$ of $\nabla$ is given by the second structure equation
$$d\theta^i_j=-\theta^i_k\wedge\theta^k_j+\Psi^i_j,$$
which gives
\begin{equation*}
\begin{split}
\Psi^1_1 & = -\frac{1}{8} \theta^2 \wedge \ov{\theta^2}, \\
\Psi^1_2 & = \frac{1}{8} \left( - \theta^1 \wedge \ov{\theta^2} + 2 \theta^2 \wedge \ov{\theta^1} - 2 \theta^1 \wedge \theta^2 - \ov{\theta^1} \wedge \ov{\theta^2} \right), \\
\Psi^2_1 & = \frac{1}{8} \left(  - \theta^2 \wedge \ov{\theta^1} + 2 \theta^1 \wedge \ov{\theta^2} +2\ov{\theta^1} \wedge \ov{\theta^2} + \theta^1 \wedge \theta^2 \right), \\
\Psi^2_2 & =  \frac{1}{8} \theta^2 \wedge \ov{\theta^2}.
\end{split}
\end{equation*}
Since $\Psi^1_1+\Psi^2_2=0$, it follows that $\mathrm{Ric}(\Omega,J)=0$. Q.E.D.
\end{proof}

We note here that the standard Ricci curvature of the Levi-Civita connection of $g$ cannot be identically zero: if it were, the metric $g$ would be almost-K\"ahler and Einstein with vanishing scalar curvature, and a result of Sekigawa \cite{S} would imply that $J$ is integrable.

\section{Further remarks and questions} \label{questions}

\noindent
(1) \ 
In \cite{TWY} it was shown that the inequality $\tilde{\Delta} u \ge - C$ holds assuming the nonnegativity of a certain tensor $\mathcal{R}$, which can be expressed in terms of the 
curvature of the canonical connection of the reference almost-Hermitian metric and the Nijenhuis tensor.  However, we cannot directly apply this result to the case of the Kodaira-Thurston manifold since the  tensor $\mathcal{R}$ associated to $(g, J)$ has negative components, as can be confirmed by a direct (and lengthy) computation.

\bigskip
\noindent
(2) \ It would be interesting to know whether Theorem \ref{maintheorem} holds for $T^2$-invariant almost complex structures on the Kodaira-Thurston manifold other than $J$.  Pushing this further, one could also investigate estimates for the Calabi-Yau equation in terms of a taming but non-compatible symplectic form.  This could be used to address the conjecture of Donaldson that the existence of a taming symplectic form implies the existence of a compatible symplectic form (see \cite{D} and also \cite{TWY}, \cite{LZ}, \cite{DLZ}).  We note  that it is of course not sensible to ask this question for our given almost complex structure $J$, since $\Omega$ is already a compatible symplectic form.

\bigskip
\noindent
(3) \ It would be desirable to remove the assumption of $T^2$ invariance in the statement of Theorem \ref{maintheorem}.  However, the inequality of Lemma \ref{lemmalap} does not seem to hold by a similar argument in this more general case and so other techniques may be needed.

\bigskip
\bigskip
\noindent
{\bf Acknowledgements}  The authors would like to express their gratitude to:  S.K. Donaldson for a number of enlightening discussions; D.H. Phong for his support and advice; 
 S.-T. Yau for 
 his encouragement and for many helpful conversations.
 In addition, the authors thank V. Apostolov and C. Taubes for some useful conversations.  Finally, the authors are especially grateful to  T. Dr\u{a}ghici for a number of very helpful suggestions, including the observation of equation (\ref{deqn2}) which simplified and improved a previous draft of this paper.

\bigskip
\noindent
Mathematics Department, Harvard University, 1 Oxford Street, Cambridge MA 02138

\bigskip
\noindent
Mathematics Department, University of California, San Diego, 9500 Gilman Drive \#0112, La Jolla CA 92093

\end{document}